\documentclass[11pt]{article}
\usepackage[english]{babel}
\usepackage{epsfig}
\usepackage{amsmath}
\usepackage{amsthm}
\usepackage{amssymb}

\newtheorem{theorem}{Theorem}[section]
\newtheorem{definition}[theorem]{Definition}

\newtheorem{proposition}[theorem]{Proposition}

\newtheorem{remark}[theorem]{Remark}

\title{\bf Function Theories in Cayley-Dickson algebras and Number Theory}

\author{
Rolf S\"{o}ren Krau{\ss}har\\
Fachbereich Mathematik\\
Erziehungswissenschaftliche Fakult\"at\\
Universit\"at Erfurt\\
Nordh\"auser Str. 63\\
99089 Erfurt, Germany\\
soeren.krausshar@uni-erfurt.de }

\begin{document}
\maketitle
\begin{abstract} In the recent years a lot of effort has been made to extend the theory of hyperholomorphic functions from the setting of associative Clifford algebras to non-associative Cayley-Dickson algebras, starting with the octonions.   An important question is whether there appear really essentially different features in the treatment with Cayley-Dickson algebras that cannot be handled in the Clifford analysis setting. Here we give one concrete example. Cayley-Dickson algebras namely admit the construction of direct analogues of CM-lattices, in particular lattices that are closed under multiplication.  Canonical examples are lattices with components from the algebraic number fields $\mathbb{Q}[\sqrt{m_1},\ldots\sqrt{m_k}]$. 
Note that the multiplication of two non-integer lattice paravectors does not give anymore a lattice paravector in the Clifford algebra. In this paper we exploit the tools of octonionic function theory  to set up an algebraic relation between different octonionic generalized elliptic functions which give rise to octonionic elliptic curves. We present formulas for the trace of the octonionic CM-division values.
\end{abstract}

\noindent {\bf Mathematical Review Classification numbers}: 11G15, 30G35

\section{Introduction}
 
 There are a number of different possibilities to generalize complex function theory to higher dimensions.
 \par\medskip\par  
 
 One classical and well-established option is to consider functions in several complex variables in $\mathbb{C}^n$ where the classical holomorphicity concept is  applied separately to each complex variable, see for example \cite{Freitag, hoermander}. From the viewpoint of algebraic geometry, the theory of several complex variables provides the natural setting to study Abelian varieties and curves.  
  \par\medskip\par 
 Another possibility is offered by Clifford analysis which considers null-solutions to a generalized Cauchy-Riemann operator that are defined in a subset of vectors or paravectors and that take values in an associative Clifford algebra, see for instance \cite{bds,dss,ghs,Kra2004}.  
  \par\medskip\par 

 In the recent years one also observes a lot of progress in extending the constructions from the setting of associative Clifford algebras to non-associative Cayley-Dickson algebras, in particular to the framework of octonions, see for example \cite{FlautSh,FL,GP,KO2018,KO2019,Nolder2018}. 
  \par\medskip\par 
 Although one has no associativity anymore, at least in the octonionic case it was possible to also generalize a number of classical theorems, such as the Cauchy integral formula or the formulas for the Taylor and Laurent series representations by following more or less the same line of argumentation as performed in Clifford analysis. See \cite{Imaeda,Nono,XL2000,XL2002,XZL}. Due to the con-associativity one has to bracket terms in a particular way together. However, apart from this, some of the results still look very similar to the formulas derived for the Clifford algebra case at least at the first glance.  
 
  \par\medskip\par 
Therefore, an important question is whether there appear really essentially different features in the treatment with the more complicated non-associative  Cayley-Dickson algebras. In \cite{KO2018} and \cite{KO2019} the authors present some important structural differences. In contrast to the Clifford analysis setting, octonionic regular functions in the sense of the Riemann approach, called $\mathbb{O}$-regular for short, do not form a left or right module anymore. If $f$ is a left $\mathbb{O}$-regular function, then it is not guaranteed that also $f \lambda$ is left $\mathbb{O}$-regular again.  So, there is no one-to-one correspondence between the set of Clifford-holomorphic functions in $\mathbb{R}^8$ and $\mathbb{O}$-regular functions.  

On the one hand, the lack of this property represents an obstacle in the development of generalizations of many other theorems to the non-associative setting.  

On the other hand, there is the challenge to figure out problems that really require a treatment with non-associative Cayley-Dickson algebras that cannot be treated in the Clifford analysis setting. 
 \par\medskip\par 
The aim of this paper to describe one very concrete example. 
 \par\medskip\par 
Cayley-Dickson algebras ${\cal{C}}_k$ offer the possibility to consider lattices $\Omega \subset {\cal{C}}_k$ that admit non-trivial left and right ideals ${\cal{L}}$ and ${\cal{R}}$ such that ${\cal{L}}\Omega \subseteq \Omega$ resp. $\Omega {\cal{R}}  \subseteq \Omega$. These are natural generalizations of classical CM-lattices, playing an important role in the treatment of algebraic points of elliptic curves, cf. \cite{Lang}.

Particularly, in Cayley-Dickson algebras one can consider lattices which, apart from their algebraic structure of a module, are additionally closed under multiplication. See also \cite{AlbuquerqueKra2008} in which some basic properties of CM-lattices in some graded $\mathbb{R}_F \mathbb{Z}^n$-algebras with deformations have been discussed, in particular for $\mathbb{R}_F \mathbb{Z}^3$-algebras and Clifford algebras. 
\par\medskip\par 
In fact, one can consider CM-lattices in Clifford algebras. However, Clifford analysis is restricted to consider functions that are only defined in the subset of paravectors and not for variables from the full Clifford algebra (apart from the particular quaternionic case). 
  
 Note that the multiplication of two non-integer lattice paravectors from $\mathbb{R}^{n+1}$ however does not give anymore a lattice paravector from $\mathbb{R}^{n+1}$. It gives an element from the full Clifford algebra that involves bivector parts. Therefore, a Clifford holomorphic function evaluating CM-values of the form $f(\mu \omega)$ with non-real multiplicators $\mu \in {\cal{L}}$ and non-real lattice paravectors $\omega \in \Omega$ cannot be defined. In the framework of Clifford analysis, the argument in the function must be a paravector again. 
 In \cite{Kra2002} we introduced a two-sided kind of CM-multiplication, considering a simultaneous multiplication of the same multiplicator from the left and from the right to the lattice of the form $f(\eta \omega \eta)$. In fact whenever $\omega,\eta$ are paravectors, then $\eta \omega \eta$ actually are paravectors again. But this is not the case, if one considers the multiplicator $\eta$ just from one side or if one considers on both sides two different multiplicators $\eta$ and $\mu$.  
 \par\medskip\par 
 This obstacle can successfully be overcome using Cayley-Dickson algebras instead. 
In constrast to the Clifford algebra, in the octonions every non-zero element  algebra is invertible. 

More generally, in the context of Cayley-Dickson algebras, one can meaningfully define functions where the arguments may stem from subsets of the full Cayley-Dickson algebra and not only from the subset of paravectors. 

 \par\medskip\par 
 
 Now, the function theory in Cayley-Dickson algebras also admits the construction of generalized Weierstra{\ss} elliptic functions which satisfy the regularity criterion in the sense of the Riemann approach. Taking special care of the non-associativity, the construction of the regular analogue of the Weierstra{\ss} $\wp$-function can be performed in exactly the same way as in Clifford analysis (cf. \cite{Kra2004},  \cite{Ry82}), namely by periodizing the partial derivatives of the associated regular Cauchy kernel and adding some convergence preserving terms. This possibility was already  roughly outlined in \cite{Kra2002,Kra2004}.  
 Recently, a $7$-fold periodic generalization of the cotangent series was explicitly written out particularly for the octonionic case in \cite{Nolder2018}. This also  underlines the current interest in this topic. 
 
 \par\medskip\par
 
 However, the treatment with Cayley-Dickson algebras provides us with a new feature, if we consider these functions associated with particular Cayley-Dickson CM-lattices, since we do not have such an algebraic structure in Clifford analysis where we are restricted to define the functions on the space of paravectors.  
 
 At least in the octonionic cases we are able to deduce explicit algebraic relations between the different CM-division values of some prototypes of generalized elliptic functions. The canonical examples of CM-lattices in Cayley-Dickson algebras are lattices whose components stem from multi-quadratic number fields $\mathbb{Q}[\sqrt{m_1},\ldots\sqrt{m_k}]$. In the octonionic case we are particularly dealing with tri-quadratic number fields. This is also in accordance with the classical complex case in which we deal with imaginary quadratic number fields of the form $\mathbb{Q}{[ \sqrt{-m}]}$.  
  \par\medskip\par
  The paper is structured as follows. 
  
  In Section~2 we summarize the most important facts and notions about hypercomplex numbers in Cayley-Dickson algebras and recall their basic properties which are used in the sequel of this paper.  
  In Section~3 we look at integral domains and introduce lattices with Cayley-Dickson multiplications. We describe them in terms of generalized integrality conditions involving norms and trace expressions. We show that lattices of the form $\mathbb{Z} + \sum_{i=1}^{2^k-1}\mathbb{Z}_i \omega_i$ where the real components of the primitive periods stem from an algebraic field of the form $\mathbb{Q}[\sqrt{m_1},\ldots\sqrt{m_k}]$ and where the elements $m_1,\ldots,m_k$ are all mutually distinct positive square-free integers, serve as important non-trivial examples of lattices with Cayley-Dickson multiplication.  
  
  In Section~4 we give a short overview about which basic tools can be carried over from Clifford analysis to the non-associative setting and explain where we meet structural obstacles. While generalizations of the Weierstra{\ss} $\zeta$-function and $\wp$-function can even be introduced in general Cayley-Dickson algebras, a number of structural and technical features require at least an alternative or composition algebra. So we turn to focus on octonions in the sequel.    
  The core piece of the paper is Section~4.3 where we establish algebraic relations between  the values of the octonionic generalized Weierstra{\ss} $\zeta$-function at a point and their octonionic CM-division values. In particular, we present an explicit algebraic formula to calculate the trace of the octonionic division values of the generalized octonionic regular Weierstra{\ss} $\wp$-function. They turn out to be elements of the field generated by the algebraic elements of the lattice components and the components of the Legendre-constants which still require an algebraic investigation in the future. 
  
  An open and very interesting question remains to ask if these division values allow us to construct Galois field extensions of the related number fields, similarly like the complex division values of associated doubly periodic function $\frac{g_2g_3}{g_2^3-27g_3^2} \wp(z)$ lie in abelian Galois field extension of $\mathbb{Q}[\sqrt{-m}]$. As a matter of fact, the construction of these kinds of algebraic field extensions was a crucial motivation for R. Fueter to develop hypercomplex function theories, cf. \cite{Fueter1948-49,Malonek}.

 \section{Some basic properties of Cayley-Dickson algebras}
 We start by introducing the construction principle of Cayley-Dickson algebras. For details, we refer the reader for instance to \cite{AlbuquerqueKra2008,Baez,Imaeda,WarrenDSmith} and elsewhere.  
 
  They contain all normed real division algebras, i.e. the fields of real and complex numbers $\mathbb{R}$ and $\mathbb{C}$, the skew field of Hamiltonian quaternions $\mathbb{H}$ and the non-associative alternative octonions $\mathbb{O}$ as special cases. 
  
 Following \cite{WarrenDSmith}, one may start with a ring $R$ that has a two-sided multiplicative neutral element $1$ and a non-necessarily commutative and non-necessarily associative multiplication. Furthermore, we impose that it has a ``conjugation'' anti-automorphism $a \mapsto \overline{a}$ with the properties that $\overline{a+b} = \overline{a}+\overline{b}$, $\overline{ab} = \overline{b} \;\overline{a}$ and $\overline{\overline{a}} = a$ for all $a,b \in R$. Then one forms pairs of numbers of the form $(a,b)$ and $(c,d)$ and defines an addition and multiplication operation by  
 $$
 (a,b)+(c,d) :=(a+c,b+d),\quad\quad (a,b)\cdot (c,d) = (ac-d\overline{b},\overline{a}d+cb). 
 $$   
 The conjugation is extended by $\overline{(a,b)}=(\overline{a},-b)$. 
 
 The simplest choice is to take for the ring $R$ the real numbers $\mathbb{R}$ in which we have $\overline{a}=a$. The above indicated doubling process, called Cayley-Dickson doubling, then generates in the first step the complex number field $\mathbb{C}$. It is the first Cayley-Dickson algebra generated by the doubling process starting with $\mathbb{R}$. If we continue performing this doubling procedure, then we obtain a chain of non-necessarily commutative nor associative algebras which are the (classical) Cayley-Dickson algebras ${\cal{C}}_k$ where $k$ denotes the step of the doubling procedure. 
 The next Cayley-Dickson algebra ${\cal{C}}_2$ is the skew field of Hamiltonian quaternions where each element can be written in the form $z=x_0 + x_ 1e_1 + x_2 e_2 + x_3 e_3$. Here, $e_i^2=-1$ for $i=1,2,3$,  $e_1 e_2 = e_3$, $e_2 e_3 = e_1$, $e_3 e_1 = e_2$ and $e_i e_j = - e_j e_i$ for all distinct $i,j$ from $\{1,2,3\}$. This algebra is not commutative anymore, but it is still associative. It still makes part of the associative Clifford algebras. This is not anymore the case after having performed the following doubling step where we arrive at the octonions $\mathbb{O}$. Octonions have the form 
 $$
 z = x_0 + x_1 e_1 + x_2 e_2 + x_3 e_3 + x_4 e_4 + x_5 e_5 + x_6 e_6 + x_7 e_7
 $$
 where $e_4=e_1 e_2$, $e_5=e_1 e_3$, $e_6= e_2 e_3$ and $e_7 = e_4 e_3 = (e_1 e_2) e_3$. 
 Like in the quaternionic case, we have $e_i^2=-1$ for all $i =1,\ldots,7$ and $e_i e_j = -e_j e_i$ for all mutual distinct $i,j \in \{1,\ldots,7\}$. The multiplication is visualized in the following multiplication table
\begin{center}
 \begin{tabular}{|l|rrrrrrr|}
 $\cdot$ & $e_1$&  $e_2$ & $e_3$ & $e_4$ & $e_5$ & $e_6$  & $e_7$ \\ \hline
 $e_1$  &  $-1$ &  $e_4$ & $e_5$ & $-e_2$ &$-e_3$ & $-e_7$ & $e_6$ \\
 $e_2$ &  $-e_4$&   $-1$ & $e_6$ & $e_1$ & $e_7$ & $-e_3$ & $-e_5$ \\
 $e_3$ &  $-e_5$& $-e_6$ & $-1$  & $-e_7$&$e_1$  & $e_2$  & $e_4$ \\
 $e_4$ &  $e_2$ & $-e_1$ & $e_7$ & $-1$  &$-e_6$ & $e_5$  & $-e_3$\\
 $e_5$ &  $e_3$ & $-e_7$ & $-e_1$&  $e_6$&  $-1$ & $-e_4$ & $e_2$ \\
 $e_6$ &  $e_7$ &  $e_3$ & $-e_2$& $-e_5$& $e_4$ & $-1$   & $-e_1$ \\
 $e_7$ & $-e_6$ &  $e_5$ & $-e_4$& $e_3$ & $-e_2$& $e_1$  & $-1$ \\ \hline 	
 \end{tabular}
\end{center}

 As one can easily deduce with the help of this table, the octonions are not  associative anymore. Therefore, they are no Clifford algebras anymore. However, one still has a number of nice properties, stemming from the fact that the octonions still form an alternative composition algebra. 
 
In particular, one has the Moufang relations, guaranteeing that $(ab)(ca) = a((bc)a)$ for all $a,b,c \in \mathbb{O}$, which particularly for $c=1$ gives the flexibility condition $(ab)a= a(ba)$. Moreover, one has the important rule 
$$
(a\overline{b})b = \overline{b}(ba) =a(\overline{b}b)=a(b \overline{b})
$$  
for all $a,b \in \mathbb{O}$. 
 
All the first Cayley-Dickson algebras ${\cal{C}}_k$ with $k \le 3$ are division algebras. 

In the next step of the Cayley-Dickson doubling we then obtain the $16$-dimensional sedenions.  As mentioned in \cite{Imaeda} one has $e_j e_k = -\delta_{jk} + \varepsilon_{jkm} e_m$ where $\delta_{ij}$ is the usual Kronecker symbol and $\varepsilon_{jkm}$ is the usual epsilon tensor, which is totally antisymmetric in its indices given by the usual permutation rule for a $3$-indexed antisymmetric tensor with values from $\{0,1,-1\}$. 
 
 Up from here we have to deal with zero-divisors (which in the case of sedenions is a measure zero subset) and  one even loses the alternative multiplication structure. Artin's theorem only guarantees the power associativity for ${\cal{C}}_k$ with $k \ge 4$. Therefore, we also lose the general Moufang identities at this level.  
 
 But an important property remains that each element $z = x_0 + \sum\limits_{j=1}^{2^k-1} x_j e_j$ of a Cayley-Dickson algebra ${\cal{C}}_k$ satisfies a quadratic equation of the form  
 $$
 z^2 - {\cal{S}}(z) z + {\cal{N}}(z) = 0
 $$
 where ${\cal{S}}(z) = z + \overline{z} = 2 x_0$ is the {\em trace} and where ${\cal{N}}(z) = z \overline{z} =|z|^2 = \sum\limits_{i=0}^{2^k-1} x_i$ is the {\em  norm} of $z$, cf. \cite{WarrenDSmith}. 
 Note that in general $dim_{\mathbb{R}}{\cal{C}}_k=2^k$. The only real normed division algebras, where one has the composition property ${\cal{N}}(zw) = {\cal{N}}(z) {\cal{N}}(w)$ are $\mathbb{R}, \mathbb{C},\mathbb{H}$ and $\mathbb{O}$. Up from the sedenions it can happen that 
 ${\cal{N}}(zw) - {\cal{N}}(z) \cdot {\cal{N}}(w) \neq 0$, see \cite{Imaeda}. 
 
 To conclude this section, we want to mention that we can get a different chain of algebras, if we construct the doubling differently, as proposed for instance in \cite{WarrenDSmith}. In the framwork of a different doubling it is possible to maintain some of the nicer properties, such as the multiplicativity of the norm.  
  
 \section{Integrality conditions and lattices with Cayley-Dickson multiplication}
In this section we introduce some number theoretical concepts. We start with
\begin{definition}
An element $z$ from a Cayley-Dickson algebra ${\cal{C}}_k$ is called rational (resp. integral) if its trace ${\cal{S}}(z)$ and its norm ${\cal{N}}(z)$ is a rational number from $\mathbb{Q}$ (or an integer from $\mathbb{Z}$ respectively).    
\end{definition}
Rational (resp. integral) Cayley-Dickson numbers form a not necessarily commutative nor associative ring, if we have 
$$
{\cal{N}}(a+b),{\cal{N}}(a\cdot b), {\cal{S}}(a+b),{\cal{S}}(a\cdot b) \in \mathbb{Q}, (\;{\rm resp.} \in \mathbb{Z}). 
$$
In the quaternionic setting, such an algebra is often called a Brandt algebra, cf. \cite{AlbuquerqueKra2008,Fueter1948-49}. Like in the quaternions, also in the octonionic setting one can easily characterize multiplicative invariant rational and integral Brandt algebras in the following way:
\begin{proposition}
Two rational (resp. integral) octonions $a,b \in \mathbb{O}$ belong to a rational (integral) non-associative Brandt algebra if and only if $2 \langle a,b \rangle$ and $2 \langle a,\overline{b} \rangle$ are elements from $\mathbb{Q}$ (resp. from $\mathbb{Z})$, where $\langle \cdot, \cdot\rangle$ is the usual Euclidean scalar product in $\mathbb{R}^8$.
\end{proposition} 
\begin{proof}
In the octonions one still has 
$$
{\cal{N}}(ab) = (ab)\cdot \overline{(ab)} =  ab (\overline{b}\overline{a})=a(b \overline{b})\overline{a} = {\cal{N}}(a) {\cal{N}}(b).
$$
Furthermore, $${\cal{N}}(a+b) = (a+b)\overline{(a+b)} = a \overline{a}+b \overline{b} + a \overline{b}+b\overline{a} = {\cal{N}}(a) + {\cal{N}}(b) + 2 \langle a,b \rangle.$$
Next, we have
$$
{\cal{S}}(ab) = ab + \overline{ab} = 2 \langle a,\overline{b}\rangle,
$$
and trivially ${\cal{S}}(a+b) = {\cal{S}}(a) + {\cal{S}}(b)$. 
\end{proof}
\begin{remark}
If the ring is additionally stable under conjugation, then the second condition $2 \langle a,\overline{b} \rangle \in \mathbb{Q}$ resp. ($\in \mathbb{Z}$) can be dropped. 
\end{remark}
 Next we introduce the concept of generalized complex multiplication of lattices in Cayley-Dickson algebras. 
 \begin{definition} (Lattices with Cayley-Dickson multiplication)\\
 	Let $\Omega_{2^k} = \mathbb{Z} \omega_0 + \mathbb{Z} \omega_1 + \cdots + \mathbb{Z} \omega_{2^k-1} $ be a $2^k$-dimensional lattice, where all elements $\omega_h$ ($h=0,\ldots2^{k-1}$) are $\mathbb{R}$-linearly independent elements from ${\cal{C}}_k$. Then we say that $\Omega_{2^k}$ has a left (right) ${\cal{C}}_{k}$-multiplication if there exists an $\eta \in \Omega_{2^k} \backslash\mathbb{Z}$ such that 
 	$$
 	\eta \cdot \Omega_{2^k} \subseteq \Omega_{2^k},\quad {\rm resp.}\quad   \Omega_{2^k} \cdot \eta \subseteq \Omega_{2^k}.
 	$$
 	In the case where we have $\eta \omega = \omega \eta$ for all $\eta,\omega \in \Omega_{2^k}$ we say that the lattice is closed under multiplication. 
 \end{definition}
 The (non-associative) ring of left multiplicators form a left (resp. right) ideal. Lattices with Cayley-Dickson multiplication can be constructed by choosing the primitive generators from a rational or integral Brandt algebra. 
 The most important examples are lattices whose components stem from multiquadratic number fields. The canonical examples can be constructed as follows. 
 
 Take $k$ mutually distinct square-free positive integers $m_1,\ldots,m_k$. 
 
 Then take a lattice of the form 
 $$
 \mathbb{Z} + \mathbb{Z} \omega_1 + \cdots + \mathbb{Z} \omega_k, \mathbb{Z} \omega_1 \omega_2 + \cdots + \mathbb{Z} \omega_{k-1} \omega_k + \cdots +  \mathbb{Z} (((\omega_1 \cdots \omega_{k-3})\omega_{k-2})\omega_{k-1})\cdot \omega_k
 $$
 where  
\begin{eqnarray*}
\omega_0 &:=& 1 \\
\omega_1 & :=& \alpha_1 \sqrt{m_1} e_1 \\
\vdots   & \vdots  & \vdots \\
\omega_k &:=& \alpha_k \sqrt{m_k} e_k \\
\omega_1 \cdot \omega_2 &:=& \alpha_1 \alpha_2 \sqrt{m_1 m_2} e_1 e_2 \\
\vdots   & \vdots  & \vdots \\
((\omega_1 \cdots )\omega_{k-1})\cdot \omega_k &:=& \alpha_1 \cdots \alpha_k \sqrt{m_1 \cdots m_k} ((e_1\cdots )e_{k-1})\cdot e_k 
\end{eqnarray*}
 and where one chooses all the appearing components $\alpha_{j_1\cdots j_r}$ to be rational numbers. 
 An arbitrary $\mathbb{Z}$ linear combination of these lattice elements then has the form 
 \begin{eqnarray*}
  & & \gamma_0  + \gamma_1 \omega_1 + \cdots + \gamma_k \omega_k +\cdots + \gamma_{1\cdots k} (\omega_1(\cdots))\omega_k \\ 
 &=& \gamma_0 + \gamma_1 \alpha_1 \sqrt{m_1} e_1 + \cdots + \gamma_k \alpha_k \sqrt{m_k}e_k + \cdots + \gamma_{1\cdots k} \alpha_{1\cdots k} \sqrt{m_1 \cdots m_k} ((e_1 \cdots)e_{k-1})e_{k}
 \end{eqnarray*} 
 where all $\gamma_{j_1\cdots j_r}$ are integers. As one easily may verify, the product of such two elements again gives an element of the same form, for instance 
 $$
 \omega_1 \cdot \omega_2 = \alpha_1 \alpha_2 \sqrt{m_1m_2} e_1 e_2.
 $$
 Note that in ${\cal{C}}_k$ one has that $(e_j e_k)e_m = \pm (e_m e_j) e_k$ which however means that the second structure constants are not anti-symmetric in general. In the case where all elements $\alpha_{j_1\cdots j_r}$ are integers, then one easily gets lattices being even closed under multiplication. As one may easily verify these lattices all form rational (resp. integral) Brandt algebras in the Cayley-Dickson algebra and they are stable under conjugation. The components of the primitive lattice generators are elements from the multiquadratic number field $\mathbb{Q}[\sqrt{m_1},\ldots,\sqrt{m_k}]$. In the particular complex case we are dealing with the classical CM-lattices of the form $\mathbb{Z} + \mathbb{Z}\tau$ where $\tau \in \mathbb{Q}[e_1\sqrt{m_1}]$. In the octonionic case we deal with tri-quadratic number fields. A bit more generally, consider eight $\mathbb{R}$ linearly independent octonionic lattice generators $\omega_h$ $(h=0,\ldots,7)$ where  
 \begin{eqnarray*}
 \omega_h &=& \alpha_{h_0} + \alpha_{h_1} \sqrt{m_1} e_1 + \alpha_{h_2} \sqrt{m_2} e_2 + \alpha_{h_3} \sqrt{m_3} e_3 \\& & + \alpha_{h_4} \sqrt{m_1m_2} e_4 + \alpha_{h_5} \sqrt{m_1m_3} e_5 + \alpha_{h_6}\sqrt{m_2m_3} e_6\\
 & &+ \alpha_{h_7} \sqrt{m_1m_2m_3} e_7,\quad\quad \alpha_{hj} \in \mathbb{Q} 
 \end{eqnarray*}
 It is easy to check that any product $\omega_h \omega_l$ turns out to be of the same form. 

For the sake of completeness, we introduce the notation $W=(\omega_{hl})_{hl}$ ($h,l \in \{0,\ldots,2^k-1\}$) for the matrix of the components of the lattice generators $\omega_h$ represented in the basis $\omega_h = \omega_{h,0} + \omega_{h,1} e_1 + \omega_{h,2} e_2 + \cdots + \omega_{h,2^k-1} (e_1(\cdots )e_{k-1})e_k$. Furthermore,  $\det(W)$ stands for its determinant and $\theta_{h,j}$ stands for the adjoint determinant associated with the elements $\omega_{h,j}$.  
  
 \section{Algebraic relations between the CM-division values of octonionic regular elliptic functions}

 \subsection{Cayley-Dickson regular functions and their basic properties}
 
 To make the paper self-contained we briefly summarize the basic facts on Cayley-Dickson regular functions in the sense of the Riemann approach and in particular on octonionic regular (monogenic) functions that are needed to prove the main results of this paper. Apart from this regularity concept, there is also the concept of slice-regularity in these algebras, cf. for instance \cite{GP}. However,  here we focus entirely on the following definition.
 \begin{definition} (Cayley-Dickson regularity) ({\rm cf. \cite{FlautSh,Imaeda}})\\
Let $U$ be an open subset in the Cayley-Dickson algebra ${\cal{C}}_k$. A function $f:U \to {\cal{C}}_k$ is called left (right) Cayley-Dickson regular, if ${\cal{D}} f(z)=0$ resp. $f(z){\cal{D}} = 0$ for all $z \in U$, where $${\cal{D}} := \frac{\partial }{\partial x_0} + \sum\limits_{j=1}^{2^k-1} \frac{\partial }{\partial x_j} e_j$$ is the generalized Cauchy-Riemann operator in the Cayley-Dickson algebra ${\cal{C}}_k$.   	
 \end{definition}
In the case $k=3$ we get the class of octonionic monogenic functions, discussed in \cite{GTBook,Nolder2018,Nono,XL2000,XL2002,XZL} which will be called $\mathbb{O}$-regular functions for short in all that follows. If $k=4$, then we deal with the sedenionc monogenic functions, see also \cite{Imaeda}. The general case has been addressed in \cite{FlautSh}. 
\par\medskip\par
All left and right Cayley-Dickson regular functions are also harmonic. They satisfy $\sum\limits_{j=0}^{2^k-1} \frac{\partial^2 f}{\partial x_i^2} = 0$. 

As important example of a function that is left and right Cayley-Dickson regular serves the generalized Cauchy kernel function $q_{\bf 0}(z) := \frac{\overline{z}}{|z|^{2^{k}}}$. Precisely speaking, it is left and right regular at any point $z\neq 0$. In the octonionic case one has $q_{\bf 0}(z) = \frac{\overline{z}}{|z|^8}$. As a direct consequence, also all partial derivatives  
$$
q_{\bf n}(z) := \frac{\partial^{|{\bf n}|}}{
	\partial x_1^{n_1} \cdots \partial x_{2^k-1}^{n_{2^k-1}}
} q_{\bf 0}(z),\quad\quad {\bf n}:=(n_1,\ldots,n_{2^k-1}),\;|{\bf n}| = \sum\limits_{j=1}^{2^k-1} n_j
$$ 
are left and right ${\cal{C}}_k$-regular at all points with $z \neq 0$.

 Following K. Imaeda, in the algebras ${\cal{C}}_k$ up from $k > 3$ one cannot set up a general Cauchy integral formula with this kernel function anymore, because the second structure constants are not anti-symmetric in these cases, cf. \cite{Imaeda}.
 
 This represents a serious obstacle. This might be a possible reason why there has not been spent that much effort to develop a comprehensive generalized function theory in the algebras ${\cal{C}}_k$ with $k > 3$ so far. 

 However, in the octonionic case, one still gets a Cauchy integral formula. Nevertheless, notice that in contrast to the Clifford analysis setting, one has to be careful with how to bracket the expressions together. From \cite{XL2000} and elsewhere we may recall
 
 \begin{theorem}(octonionic Cauchy-integral formula)\\
Let $U \subseteq \mathbb{O}$ be open and $K \subset U$ be a $7$-dimensional compact set with an orientable strongly Lipschitz boundary $\partial K$. Let $f:U \to \mathbb{O}$ be a left (right) $\mathbb{O}$-regular function. Then,
$$
f(z) = \frac{1}{\omega_{8}} \int\limits_{\partial K} q_{\bf 0}(z-w) \cdot \Bigg( d\sigma(w) f(w)\Bigg),\quad {\rm resp.}\quad f(z) = \frac{1}{\omega_{8}} \int\limits_{\partial K} \Bigg(f(w) d\sigma(w)\Bigg) \cdot q_{\bf 0}(z-w)
$$        	
where $\omega_{8} = \frac{\pi^4}{3}$ is the surface measure of the unit hypersphere in $\mathbb{O}$. 
 \end{theorem} 

Following \cite{Imaeda}, even in the non-alternative cases, every function $f:U \to {\cal{C}}_k$ that is left (right) ${\cal{C}}_k$-regular in an open neighborhood around a point $a \in U$ can locally be expanded in a Taylor series of the form    
$$
f(z) = \sum\limits_{|{\bf n}|=0}^{+\infty} V_{\bf n}(z-a) a_{\bf n}, \quad {\rm resp.} \quad f(z) = \sum\limits_{|{\bf n}|=0}^{+\infty} a_{\bf n} V_{\bf n}(z-a), 
$$
where $a_{\bf n} := \frac{\partial^{|{\bf n}|}}{\partial {\bf x}^{\bf n}} f(a)$ are hypercomplex numbers from ${\cal{C}}_k$ and the polynomials $V_{\bf n}(z)$ are the generalized Fueter polynomials. This is a consequence of the power associativity that remains valid in all Cayley-Dickson algebras. 

In the Cayley-Dickson algebra setting, the Fueter polynomials have the form 
$$
V_{\bf n}(z) = \frac{1}{|{\bf n}|!} \sum\limits_{\pi perm({\bf n})} (Z_{\pi(n_1)}(Z_{\pi}(n_2)( \cdots (Z_{\pi(n_{2^k-1})} Z_{\pi(n_{2^k})})\cdots))).
$$
Here, $perm({\bf n})$ denotes the set of all distinguishable permutations of the sequence $(n_1,n_2,\ldots,n_{2^k-1})$ and $Z_i := V_{\tau(i)}(z) := x_i - x_0 e_i$ for all $i=1,\ldots,2^{k}-1$, see \cite{XL2002} Theorem C p.208, where the octonionic case has been treated in specifically. 

Like in the complex and Clifford analysis setting, Cauchy's integral formula allows us easily to show the particular octonionic case the following 
\begin{theorem} (octonionic Liouville's theorem)\\
	If $f: \mathbb{O} \to \mathbb{O}$ is left or right $\mathbb{O}$-regular and bounded over the whole algebra $\mathbb{O}$, then $f$ must be a constant.  
\end{theorem}
\begin{proof} We describe the left $\mathbb{O}$-regular case. 
By performing partial differentiation on the octonionic Cauchy's integral formula, one directly obtains that 
$$
\frac{\partial }{\partial x_i} f(z) = \frac{1}{\omega_{8}} \int\limits_{|z-w|=r} q_{\tau(i)}(z-w)\cdot \Big(d\sigma(w) \cdot f(w)\Big)
$$
Thus, in view of  $|q_{\tau(i)}(z-w)| \le M |z-w|^{-8}$ (with a real constant $M$) and since the measure of the surface of a $8$-dimensional ball of radius $r$ is  $\frac{\pi^{4} r^{7}}{3}$, we have  
$$
\Big|\frac{\partial }{\partial x_i} f(z)\Big| \le \frac{M}{r} \sup_{z \in \mathbb{O}}\{|f(z)|\}
$$
which tends to zero, because $\sup_{z \in \mathbb{O}}\{|f(z)|\}$ is bounded. Hence, $f$ must be constant. 
\end{proof}
Note that we have used the Cauchy integral formula which is not available for higher dimensional Cayley-Dickson algebras. Therefore, this simple proof cannot be extended directly to the more general algebras ${\cal{C}}_k$ for $k>3$.  
\par\medskip\par
As a simple consequence on can also establish that even every function $f:\mathbb{O} \to \mathbb{O}$ that is harmonic and bounded  over the whole algebra $\mathbb{O}$ is a constant.
\par\medskip\par
 In particular, for the octonionic case one may also introduce in view of the property of being a composition algebra: 
 \begin{definition} (Octonionic meromorphicity)\\
 	Let $U \subseteq \mathbb{O}$ be an open set. Suppose that $a \in U$ and that $f:U \backslash\{a\} \to \mathbb{O}$ is left (right) $\mathbb{O}$-regular. Then the point $a$ is called a non-essential isolated singularity of $f$, if there exists a non-negative integer $n$ such that $|a|^n |f(z)|$ remains bounded in a  neighborhood around $a$. More generally, let $S \subset U$ be a closed subset with an orientable boundary. If $f:U \backslash S \to \mathbb{O}$ is left (right) $\mathbb{O}$-regular, then we say that $S$ is a non-essential singular set of $f$, if there exists a non-negative integer $n$ such that the expression $\rho^n |f(z)|$ remains bounded where $\rho :=\sup_{s \in S}\{|s-z|\}$. Left (right) $\mathbb{O}$-regular functions that have at most unessential singular points in a closed subset $S \subset \mathbb{O}$ are called left (right) $\mathbb{O}$-meromorphic.        
 \end{definition}

 \subsection{Basic properties of octonionic regular elliptic functions}
 
 In this subsection we briefly summarize the most basic properties of octonionic regular elliptic functions. We start by giving its definition.
 \begin{definition}
 	Let $\Omega_8 = \mathbb{Z} \omega_0 + \ldots + \mathbb{Z} \omega_7$ be an arbitrary eight-dimensional octonionic lattice; that means that $\omega_0,\ldots,\omega_7$ are supposed to be eight $\mathbb{R}$-linearly  independent octonions. Further, let $S \subset \mathbb{O}$ be a closed subset. 
 	A left (right) $\mathbb{O}$-regular function $f:\mathbb{O} \backslash S \to \mathbb{O}$ that has atmost unessential singularities at the points of $S$  	
 	 is called a left (right) $\mathbb{O}$-regular {\em elliptic} function, if it satisfies at each $z \in \mathbb{O} \backslash S$ that $f(z+\omega) = f(z)$ for all $\omega \in \Omega_8$ and $S+\omega = S$ for all $\omega \in \Omega_8$.  
 \end{definition}
 This is the same definition as given for the case of quaternions in \cite{Fueter1948-49} and for the case of paravector-valued functions with values in associative Clifford algebras in \cite{Kra2004}. 
 \par\medskip\par
 \begin{remark}
A crucial difference to the Clifford analysis setting consists in the fact that the  set of left (right) $\mathbb{O}$-regular elliptic functions is not a right (left) non-associative $\mathbb{O}$-module, see {\rm \cite{KO2019} Remark 4.4}. But one still has the property that left (right) $\mathbb{O}$-regularity is inherited by partial derivation of such a function.  
\end{remark}
 
 Like in the complex and Clifford-holomorphic case (cf. \cite{Kra2004,Ry82}) it is not possible to find any non-constant $\Omega_8$-periodic function that is left or right $\mathbb{O}$-regular on the whole algebra $\mathbb{O}$. This is due to the fact that the topological quotient $\mathbb{O}/\Omega_8$ is a compact $8$-torus. However, any function that is $\mathbb{O}$-regular on the whole algebra $\mathbb{O}$ is continuous in particular. Hence, it is bounded on the topological quotient torus, which means that it is bounded on the closure of each period cell. Therefore, as a consequence of the generalized octonionic Liouville theorem, such a function must be a constant. The same holds under the weaker condition of being harmonic.  
 \par\medskip\par
 Therefore, a non-constant $\mathbb{O}$-regular or harmonic $\Omega_8$-periodic function must have singularities. 
  \par\medskip\par
 The simplest non-trivial examples of $\mathbb{O}$-regular functions  are given in terms of the $\Omega_8$-periodization of the partial derivatives of the octonionic Cauchy kernel function $q_{\bf 0}(z)$. Applying the classical Eisenstein series convergence argument, the series 
 $$
\wp_{\bf n}(z) := \sum\limits_{\omega \in \Omega_8} \Big[\Big(\frac{\partial^{|{\bf n}|} }{\partial {\bf x}^{\bf n}} q_{\bf 0} \Big)(z + \omega)\Big]
 $$
 converge normally whenever $|{\bf n}| :=\sum_{j=1}^7 n_j \ge 2$, since $\sum\limits_{\omega \in \Omega_8 \backslash\{0\}} |\omega|^{-7+\alpha}$ is convergent if and only if $\alpha > 1$. In the limit case where ${\bf n}=\tau(i)$ is a multi-index of length $1$ (where $n_j = \delta_{ij}$ for one particular $i \in \{1,\ldots,7\}$) the series $\sum_{\omega \in \Omega_8} q_{\tau(i)}(z+\omega)$ where $q_{\tau(i)}(z)= \frac{\partial }{\partial x_i} q_{\bf 0}(z)$ is not convergent anymore. However, as usually, convergence can be achieved by adding a convergence preserving term in the way 
 $$
 \wp_{\tau(i)}(z) = q_{\tau(i)}(z) + \sum\limits_{\omega \in \Omega_8 \backslash\{0\}} \Big(q_{\tau(i)}(z+\omega)-q_{\tau(i)}(\omega)\Big),
 $$
 similarly to the Clifford algebra case, cf. \cite{Fueter1948-49,Kra2004,Ry82}.

In fact, in complete analogy to the Clifford analysis setting, the series  $\wp_{\tau(i)}(z)$ is $\Omega_8$-periodic. 

If we consider the subseries summing only over the lattice points of a seven-dimensional sublattice $\Omega_7$ with ${\cal{S}}(\Omega_7) = 0$, then we  deal with a partial derivative of the octonionic generalized cotangent function from \cite{Nolder2018}.

The left  $\mathbb{O}$-regular primitive of the fully $\Omega_8$-periodic function $\wp_{\tau(i)}$ given by 
$$
\zeta(z) := q_{\bf 0}(z) + \sum\limits_{\omega \in \Omega_8 \backslash\{0\}}
\Big(
q_{\bf 0}(z+\omega) - q_{\bf 0}(\omega) + \sum\limits_{j=1}^7 V_{\tau(i)}(z) q_{\tau(i)}(\omega)  
\Big)
$$
provides us with the direct analogue of the Weierstra{\ss} $\zeta$-function, satisfying like in the Clifford case $\frac{\partial \zeta}{\partial x_i} = \wp_{\tau(i)}$. $\zeta(z)$ is not $\Omega_8$-periodic anymore. However, it is quasi $\Omega_8$-periodic of the  form $$
\zeta(z+\omega_h) - \zeta(z) = \eta_h,
$$
where $\eta_h$ are octonionic constants, the so-called octononic Legendre constants. In the latter equation $\omega_h$ ($h=0,\ldots,7$) represent the primitive periods of $\Omega_8$. The Legendre constants are given by 
$$
\eta_h = \zeta(-\frac{\omega_h}{2}+\omega_h) - \zeta(-\frac{\omega_h}{2}) = 2 \zeta(\frac{\omega_h}{2}),
$$ 
because $\zeta$ is an odd function which can readily be seen by a rearrangement of the series. 

\begin{remark}
The left $\mathbb{O}$-regularity of $\zeta(z)$ follows from the application of Weierstra{\ss} convergence theorem from {\rm \cite{XL2002}~Theorem~11} to each particular term of the series. Note that 
\begin{eqnarray*}
{\cal{D}}[V_{\tau(j)(z) q_{\tau(i)}(\omega)}] &= & \frac{\partial }{\partial x_0} \Big(x_j-x_0e_j\Big) q_{\tau(i)}(\omega) + e_j \Big[\frac{\partial }{\partial x_j} 
\Big(x_j-x_0e_j\Big) q_{\tau(i)}(\omega) \Big]\\
& & + \sum\limits_{i \neq 0,j} e_j \Big[\frac{\partial }{\partial x_i} 
\Big(x_j-x_0e_j\Big) q_{\tau(i)}(\omega)\Big] \\
& = & -e_j q_{\tau(i)}(\omega) + e_j q_{\tau(i)}(\omega) = 0.
\end{eqnarray*}
The left and right $\mathbb{O}$-regularity of the functions $\wp_{\tau(i)}$ and $\wp_{\bf n}$ with $|{\bf n}| \ge 2$ is evident.	
\end{remark}

The detailed convergence proof follows along the same lines as in the Clifford case. Hence, we omit it. See \cite{Kra2004} and see also \cite{Nolder2018} for the particular octonionic cotangent series constructions treated there. 

\begin{remark}
The same series constructions can also be made in the general $2^k$-dimensional Cayley-Dickson algebras by inserting the functions $q_{\bf 0}(z) := \frac{\overline{z}}{|z|^{2^k}}$ or their partial derivatives, respectively, in the series constructions. The convergence conditions remain the same. However, notice that we do not have an analogue of Liouville's theorem in the more general setting due to the lack of a direct analogue of Cauchy's integral formula. Anyway, the Taylor expansion representation of the Cauchy kernel function and its partial derivatives which are required in the convergence proof remain valid. 
\end{remark}    
 
 \subsection{Octonionic regular elliptic functions for generalized CM-lattices and their division values}
 
 The results presented in this subsection only address the octonionic setting, because the proofs explicitly use the alternative property of the octonions which is lost in the higher dimensional Cayley-Dickson algebras ${\cal{C}}_k$ with $k > 3$. 
 \par\medskip\par   
 Before we start, we need to establish some important preparatory statements. 
 
 First we note that 
 \begin{eqnarray*}
 	q_{\bf 0}(\lambda \cdot (z\mu) ) &=& \frac{\overline{\lambda \cdot (z \mu)}}{|\lambda \cdot (z \mu)|^8} = \frac{\overline{z\mu}\cdot\overline{\lambda}}{|\lambda|^8 |z \mu|^8}\\
 	&=& \frac{(\overline{\mu} \;\overline{z})\cdot \overline{\lambda}}{|\mu|^8 |z|^8 |\lambda|^8} = \Big(q_{\bf 0}(\mu) q_{\bf 0}(z) \Big)\cdot q_{\bf 0}(\lambda)	
 \end{eqnarray*} 
 where we used the conjugation property $\overline{ab} = \overline{b}\overline{a}$. 
 
 Now we want to show the following formula
 \begin{proposition}\label{Prp49}
 	For each $\lambda,\mu, z \in \mathbb{O} \backslash\{0\}$ we have 
 	\begin{equation} \label{q0trafo}
 	\Big(\mu\cdot q_{\bf 0}(\lambda \cdot (z \mu)) \Big)\cdot \lambda = \frac{1}{({\cal{N}}(\mu) {\cal{N}}(\lambda))^3 }q_{\bf 0}(z).
 	\end{equation}
 \end{proposition}
 \begin{proof} Here, we have to argue very carefully, since we do not have associativity. So, we do it step by step. First we note that 
 	$$
 	\mu \cdot \Big(q_{\bf 0}(z \mu) \Big) = \mu \cdot \Bigg(   
 	\frac{\overline{z \mu}}{|z\mu|^8}\Bigg) = 
 	\mu \cdot \Bigg( \frac{\overline{\mu} \;\overline{z}}{|\mu|^8|z|^8}\Bigg) 
 	= \frac{(\mu\;\overline{\mu}) \overline{z}}{|\mu|^8|z|^8} = \frac{1}{|\mu|^6} q_{\bf 0}(z)
 	$$
 	where we exploited the alternating property $\mu(\overline{\mu}\; \overline{z}) = (\mu \overline{\mu})(\overline{z})$. Here, we exploited a special property of the octonions that cannot be extended to sedenions or the following  Cayley-Dickson algebras. 
 	
 	Next we put $f(z):= \mu \cdot \Big(q_{\bf 0}(z \mu)   \Big) = \frac{1}{|\mu|^6} q_{\bf 0}(z)$. And now we can conclude that  
 	\begin{eqnarray*}
 		\Bigg(\mu \cdot q_{\bf 0}\Big(\lambda\cdot (z\mu) \Big)    \Bigg) \cdot \lambda &=& f(\lambda z) \cdot \lambda \\
 		& = & \frac{1}{|\mu|^6} q_{\bf 0}(\lambda z) \cdot \lambda \\
 		& = & \frac{1}{|\mu|^6} \Bigg(\frac{\overline{\lambda z}}{|\lambda z|^8}\Bigg)\cdot \lambda \\
 		& = & \frac{(\overline{z}\;\overline{\lambda})\cdot \lambda}{|\mu|^6|z|^8|\lambda|^8}	
 		= \frac{\overline{z} \cdot(\overline{\lambda}\cdot \lambda)}{|\mu|^6|z|^8|\lambda|^8}\\
 		&=& \frac{1}{|\mu|^6} \frac{1}{|\lambda|^6} q_{\bf 0}(z) = \frac{1}{({\cal{N}}(\mu) {\cal{N}}(\lambda))^3 }q_{\bf 0}(z).
 	\end{eqnarray*}
 \end{proof}
 \begin{remark}
  We wish to emphasize clearly that this formula holds for octonions. Our argumentation cannot be extended beyond octonions in the next steps of the usual Cayley-Dickson doubling,  at least not by using this chain of arguments, because we explicitly used the alternative property. 	
 \end{remark}

 As a direct consequence of the fact that $\frac{1}{({\cal{N}}(\mu) {\cal{N}}(\lambda))^3 }$ is real-valued  we obtain the following important statement
 \begin{proposition}\label{reg}
 For all $\lambda,\mu \in \mathbb{O} \backslash\{0\}$ the functions $	\Bigg(\mu \cdot q_{\bf 0}\Big(\lambda\cdot (z\mu) \Big)    \Bigg) \cdot \lambda$ and in particular $q_{\bf 0}(\lambda z) \cdot \lambda$ and $\mu\cdot q_{\bf 0}(z\mu)$ are left and right $\mathbb{O}$-regular at each $z \in \mathbb{O} \backslash\{0\}$. 
 \end{proposition}
\begin{remark}
	 Note again that in contrast to the Clifford analysis setting the property of the latter proposition is not immediate, because left (right) $\mathbb{O}$-regular functions do not form a right (left) $\mathbb{O}$-module as mentioned before. The property is true for the particular function $q_{\bf 0}$ but at least as far as we know it is not evident for any arbitrary left (right) $\mathbb{O}$-regular function $f$.  
\end{remark}
  
 Now let particularly $\Omega_8 = \mathbb{Z} + \mathbb{Z} \omega_1 + \cdots + \mathbb{Z}\omega_7$ ($\omega_0=1$) be a lattice with octonionic multiplication as defined in Section~3.  
 Suppose that $\lambda \in \mathbb{O} \backslash \mathbb{R}$ is a non-trivial multiplicator from a left ideal ${\cal{L}}$ with ${\cal{L}} \Omega \subseteq \Omega$ and assume that $\mu \in \mathbb{O} \backslash \mathbb{R}$ is a non-trivial multiplicator from a right ideal ${\cal{R}}$ with $\Omega {\cal{R}} \subseteq \Omega$. 
 
 We look at the associated octonionic left $\mathbb{O}$-regular Weierstra{\ss} $\zeta$-function
 $$
 \zeta(z) = \zeta(z,\Omega) = q_{\bf 0}(z) + \sum\limits_{\omega \in \Omega_8 \backslash\{0\}} \Big[
 q_{\bf 0}(z+\omega) - q_{\bf 0}(\omega) + \sum\limits_{j=1}^7  V_{\tau(j)}(z) \Big(q_{\tau(j)}(\omega) \Big)
 \Big].
 $$
 
 If $\Omega_8$ is such a lattice and ${\cal{L}}$ such a left ideal, then we can find a $\lambda \in {\cal{L}} \backslash \mathbb{R}$ such that $\lambda \omega \in \Omega_8$ for all $\omega \in \Omega_8$. 
 
 So, the function $\zeta(\lambda z) \cdot \lambda$ is a well-defined quasi-elliptic function on the same lattice $\Omega_8$, since $\lambda \omega \in \Omega_8$ for all $\omega \in \Omega_8$. It is easy to see that it is at least harmonic when applying Weierstra{\ss} convergence theorem to each term of the series. 
 According to  Proposition~\ref{reg} the term $q_{\bf 0}(\lambda z) \lambda$ is   left and right $\mathbb{O}$-regular and hence harmonic. The terms $V_{\tau(j)}(z) \Big(q_{\tau(j)}(\omega)\Big)$ are all linear and therefore in the kernel of the Laplacian. Since the Laplacian is a scalar operator, one has $\Delta 
 [ q_{\bf 0}(\lambda z + \omega) \lambda] = \Delta 
 [q_{\bf 0}(\lambda z + \omega)]\lambda$. So since $0 = \Delta 
 [q_{\bf 0}(\lambda z) \lambda]$ and since $\mathbb{O}$ is a division algebra one also  has $\Delta [q_{\bf 0}(\lambda z)] = 0$. Since $\omega \neq 0$ one can find a $t \in \mathbb{O}$ such that $\omega = \lambda t$. Applying a linear shift in the  argument also leads to the fact that $\Delta[q_{\bf 0}(\lambda z+\omega)] =  \Delta[q_{\bf 0}(\lambda z+\lambda t)] = \Delta[\Big(q_{\bf 0}(z + t)\Big)\lambda] = 0$, since the differential remains invariant under the shift $z+t$.

 Analogously, there are elements $\mu \in {\mathcal{R}} \backslash \mathbb{R}$ such that $\omega \mu \in \Omega_8$ for all $\omega \in \Omega_8$ so that the function $\mu \cdot \zeta(z \mu)$ is also a well-defined quasi-elliptic function again on the same lattice, since also $\omega \mu \in \Omega_8$ for all $\omega \in \Omega_8$. Here again, we can establish that this function is harmonic at least. 
 
 More generally, and bearing in mind the non-associativity, the {\em two} functions 
 $$
 \zeta^1_{\lambda,\mu}(z) := \Big( \mu \cdot \zeta(\lambda\cdot(z\mu)     )  \Big) \cdot \lambda
 $$
 and
 $$
  \zeta^2_{\lambda,\mu}(z) := \mu \cdot \Big(\zeta((\lambda z)\cdot \mu) \cdot \lambda  \Big)
 $$
are well-defined at least harmonic quasi-elliptic functions for all $\lambda \in {\cal{L}}, \mu \in {\cal{R}}$ since 
$$
(\lambda \Omega_8) \cdot \mu \subseteq \Omega_8 \mu \subseteq \Omega_8
$$ 
and 
$$
\lambda \cdot (\Omega_8 \mu) \subseteq \lambda \Omega_8 \subseteq \Omega_8
$$
in view of the CM-property. 

Completely analogously,  one can conclude in the same way that $\zeta^2_{\lambda,\mu}(z)$ is at least harmonic, too.

Note that both functions in general differ from each other as a consequence of the lack of associativity. For the sake of convenience we focus up from now on the function $\zeta_{\lambda,\mu}(z):=\zeta^1_{\lambda,\mu}(z)$ since the other version can be treated analogously. 

The function $\Big( \mu \cdot \zeta(\lambda\cdot(z\mu))\Big)\cdot \lambda$ is singular if and only if $\lambda \cdot (z \mu) = \omega$ for a lattice point $\omega \in \Omega_8$. This is equivalent to $(z\mu)=\lambda^{-1}\omega \Longleftrightarrow z = (\lambda^{-1}\omega)\cdot \mu^{-1}$. Thus, the function $\zeta^1_{\lambda,\mu}(z)$ has isolated point singularities at exactly the points $z = (\lambda^{-1}\omega)\cdot \mu^{-1}$ where $\omega$ runs through $\Omega_8$. 

By a counting argument we obtain the $\zeta_{\lambda,\mu}$ has exactly ${\cal{N}}(\lambda \mu)^4$-many isolated point singularities in the fundamental perioc cell
$$
{\cal{F}}:=\{x = \alpha_0 + \alpha_1 \omega_1 + \cdots +\alpha_7 \omega_7 \mid 0 \le \alpha_j < 1,\;j\in \{0,\ldots,7\}\},
$$
which in particular contains $0$.  

The set of all these singularities that lie in the fundamental set will be denoted by ${\cal{V}}_{\lambda,\mu;\Omega}$ in all that follows. 

\begin{remark}
	In the simple case where $\lambda=2$ and $\mu=1$ this set consists exactly of those points where the coordinates either have the value zero or $1/2$, whose cardinality evidently equals $2^8 ={\cal{N}}(2)^4$.     
\end{remark}

Now we have all pre-requisites to formulate and prove our main theorem
\begin{theorem}
Let $\Omega_8 = \mathbb{Z} + \mathbb{Z} \omega_1 + \cdots + \mathbb{Z} \omega_7$, $ (\omega_0:=1)$ be an octonionic lattice with octonionic multiplication with the properties and notations as described previuously. The trace of the octonionic division values of the $\mathbb{O}$-regular elliptic functions $\wp_{\tau(i)}$, ($i = 1,\ldots,7$) can be expressed by   
$$
\sum\limits_{v \in {\cal{V}}_{\lambda,\mu;\Omega} \backslash\{0\}} \wp_{\tau(i)}(v) 
= - \frac{({\cal{N}}(\lambda) {\cal{N}}(\mu)  )^3  }{\det(W)} \Bigg[
\sum\limits_{h=0}^7 \Theta_{h_i}\Big( \mu \cdot (\sum\limits_{j=0}^7 n_{h_j} \eta_j)\cdot \lambda
\Big)
- {\cal{N}}(\lambda) {\cal{N}}(\mu) \sum\limits_{h=0}^7\Theta_{h_i} \eta_h
\Bigg],
$$	
where $\Theta_{hi}$ denotes the adjoint determinant associated with the lattice component $\omega_{hi}$ and where ${\cal{V}}_{\lambda,\mu;\Omega} :=\{v \in {\cal{F}} \mid v = (\lambda^{-1}\omega)\cdot\mu^{-1},\;\omega \in \Omega_8\}$.  
\end{theorem}
\begin{proof} 
As a consequence of formula~(\ref{q0trafo}) we may infer that the Laurent expansion of the function $\zeta_{\lambda,\mu}$ centered at zero has the form 
$$
\zeta_{\lambda,\mu}(z) = \Bigg(\mu \zeta(\lambda\cdot(z \mu))\Bigg)\cdot \lambda = \frac{1}{({\cal{N}}(\mu) {\cal{N}}(\lambda))^3} q_{\bf 0}(z)+ A(z),
$$
where $A$ is a function that is at least harmonic in some neighborhood of $0$. 

Conversely, the function 
$$
\sum\limits_{v \in {\cal{V}}_{\lambda,\mu;\Omega}} \zeta(z + v)
$$
can be written in the form $q_{\bf 0}(z) + B(z)$, where also $B(z)$ is a function that is definetely left $\mathbb{O}$-regular around $0$. Therefore, the difference function
\begin{equation}\label{difference}
f(z) := \zeta_{\lambda,\mu}(z)- \frac{1}{({\cal{N}}(\mu) {\cal{N}}(\lambda))^3} \cdot \sum\limits_{v \in {\cal{V}}_{\lambda,\mu;\Omega}} \zeta(z+v)
\end{equation}
is at least harmonic around $0$, too. Similarly, one gets the same result if one considers the Laurent expansion around another singular point $v \in {\cal{V}}_{\lambda,\mu;\Omega}$. 

The same is true for all the partial derivatives $(i \in \{ 0,\ldots,7\})$
\begin{eqnarray}\label{diffdifference}
\frac{\partial }{\partial x_i} f(z) &=& \frac{\partial }{\partial x_i} \zeta_{\lambda,\mu}(z) - \frac{1}{({\cal{N}}(\mu) {\cal{N}}(\lambda))^3} \sum\limits_{v \in {\cal{V}}_{\lambda,\mu;\Omega}} \frac{\partial }{\partial x_i} \zeta(z+v) \nonumber \\
& = & \frac{\partial }{\partial x_i} \zeta_{\lambda,\mu}(z) - \frac{1}{({\cal{N}}(\mu) {\cal{N}}(\lambda))^3} \sum\limits_{v \in {\cal{V}}_{\lambda,\mu;\Omega}} \wp_{\tau(i)}(z+v). 
\end{eqnarray}
 
Since the functions $\frac{\partial }{\partial x_i} \zeta_{\lambda,\mu}(z)$ and $\wp_{\tau(i)}(z+v)$ are all $\Omega_8$-periodic, each function $f_i := \frac{\partial f}{\partial x_i}$ is $\Omega_8$-periodic and  must be at least  harmonic on the entire algebra $\mathbb{O}$, since it has no singularity inside of  ${\cal{V}}_{\lambda,\mu;\Omega}$. So, in view of Liouville's theorem. there are octonionic constants $C_i \in \mathbb{O}$ such that $f_i(z) = C_i$ for all $z$. Now we make the following ansatz
\begin{equation}\label{zetasum}
\zeta_{\lambda,\mu}(z)= \frac{1}{({\cal{N}}(\mu) {\cal{N}}(\lambda))^3} \sum\limits_{v \in {\cal{V}}_{\lambda,\mu;\Omega}} \zeta(z+v) + \sum\limits_{j=1} V_{\tau(j)}(z) C_j + C
\end{equation}
where $C$ is a further octonionic constant.  
\par\medskip\par 
Now let $\omega_h$ be a primitive period of $\Omega_8$ ($h=0,1,\ldots,7$). Then we have
\begin{eqnarray*}
\zeta_{\lambda,\mu}(z + \omega_h) &=& \Bigg(\mu \cdot \zeta\Big(\lambda \cdot((z+\omega)\mu)  \Big)   \Bigg) \cdot \lambda \\
& = & \frac{1}{({\cal{N}}(\mu) {\cal{N}}(\lambda))^3} \sum\limits_{v \in {\cal{V}}_{\lambda,\mu;\Omega}} \zeta(z+\omega_h+v) + \sum\limits_{j=1}^7  V_{\tau(j)}(z+\omega_h) C_j + C
\end{eqnarray*}	
for all $h \in \{0,1,\ldots,7\}$. 
\par\medskip\par
Now the crucial aspect is that the lattice $\Omega_8$ has octonionic multiplication of the form $(\lambda \Omega_8) \mu \subseteq \Omega_8$ and $\lambda(\Omega_8 \mu) \subseteq \Omega$. Therefore, there exist integers $n_{h_j} \in \mathbb{Z}$ such that
\begin{equation}\label{CMproperty}
\lambda\cdot (\omega_h \mu) = \sum\limits_{j=0}^7 n_{h_j} \omega_j.
\end{equation} 
In view of the Legendre relation that we stated in the previous subsection we have the additive relation 
\begin{equation}\label{CMLegendre}
\zeta(\lambda(z \mu)+ \lambda(\omega_h \mu)) = \zeta(\lambda\cdot(z \mu)) + \sum\limits_{j=0}^7 n_{h_j} \eta_j 
\end{equation}
with the octonionic Legendre constants $\eta_0,\ldots,\eta_7$. 

Applying formula~(\ref{zetasum}) we get 
$$
\Bigg(\mu\cdot \zeta\Big(\lambda( (z+\omega_h)\cdot \mu        )\Big)\Bigg)\cdot \lambda	 =  \frac{1}{({\cal{N}}(\mu) {\cal{N}}(\lambda))^3} \sum\limits_{v \in {\cal{V}}_{\lambda,\mu;\Omega}} \zeta(z+\omega_h+v) 
 + \sum\limits_{j=1}^7 V_{\tau(j)}(z+\omega_h) C_j + C	
$$
Using the Legendre relation, the latter equation is equivalent to
\begin{eqnarray*}
\Bigg(\mu\cdot \zeta\Big( \lambda\cdot (z \mu)+\lambda\cdot (\omega_h \mu)       \Big)\Bigg)\cdot \lambda	& = & \frac{1}{({\cal{N}}(\mu) {\cal{N}}(\lambda))^3} \sum\limits_{v \in {\cal{V}}_{\lambda,\mu;\Omega}}\Big[ \zeta(z+v) +\eta_h \Big]\\
& & + \sum\limits_{j=1}^7 V_{\tau(j)}(z+\omega_h) C_j + C. 
\end{eqnarray*}
Applying~(\ref{CMLegendre}) to the previous equation leads to 
\begin{eqnarray*}
\Bigg(\mu\cdot \zeta\Big(\lambda\cdot(z \mu) \Big)\Bigg)\cdot \lambda	+ \Big(\mu(\sum_{j=0}^7 n_{h_j} \eta_j)\Big)\cdot \lambda & = & 
	 \frac{1}{({\cal{N}}(\mu) {\cal{N}}(\lambda))^3} \sum\limits_{v \in {\cal{V}}_{\lambda,\mu;\Omega}}\Big[ \zeta(z+v) +\eta_h \Big]\\
	 & & + \sum\limits_{j=1}^7  V_{\tau(j)}(z) C_j + C \\
	 & & + \sum\limits_{j=1}^7  V_{\tau(j)} (\omega_h) C_j.
\end{eqnarray*}
Since the first term of the left-hand side equals the expression of the sum of the first two terms of the right-hand side in view of~(\ref{zetasum}), we obtain the relation
$$
\Big( \mu \cdot (\sum\limits_{j=0}^7 n_{h_j}\eta_j )\Big)\cdot \lambda = \frac{1}{({\cal{N}}(\mu) {\cal{N}}(\lambda))^3} \sum\limits_{v \in {\cal{V}}_{\lambda,\mu;\Omega}} \eta_h + \sum\limits_{j=1}^7  V_{\tau(j)}(\omega_h) C_j.
$$
Next, since the cardinality of ${\cal{V}}_{\lambda,\mu;\Omega}$ equals ${\cal{N}}(\lambda \mu)^4$, we have 
$\sum_{v \in {\cal{V}}_{\lambda,\mu;\Omega}} \eta_h = {\cal{N}}(\lambda \mu)^4 \cdot \eta_h$.  
Thus, we arrive at 
\begin{equation}\label{stern1}
\Big(\mu \cdot (\sum\limits_{j=0}^7 n_{h_j} \eta_j)  \Big)\cdot \lambda = {\cal{N}}(\lambda \mu) \eta_h + \sum\limits_{j=1}^7 (\omega_{h_j} - e_j \omega_{h_0}) C_j,
\end{equation}
where we put $\omega_h = \sum\limits_{j=0}^7 \omega_{h_j} e_j$ for the representation of the primitive periods $\omega_h$ in the coordinates of the canonical basis elements $e_0,e_1,\ldots,e_7$. 
\par\medskip\par 
Next, let us write $\Theta_{h_j}$ for the adjoint determinant associated with the element $\omega_{h_j}$. Then, classical linear algebra tells us that 
\begin{equation}\label{stern2}
\sum\limits_{h=0}^7 \omega_{h_i} \Theta_{h_l} = \delta_{il} \det(W).
\end{equation}
Combining this formula with~(\ref{stern1}), we obtain that  
$$
\sum\limits_{h=0}^7 \Big[\Theta_{h_i}\Big(\mu \cdot (\sum\limits_{j=0}^7 n_{h_j} \eta_j)  \Big)\cdot \lambda  \Big] - {\cal{N}}(\lambda) {\cal{N}}(\mu) \sum\limits_{h=0}^7 \Theta_{h_i} \eta_h = \sum\limits_{h=0}^7 \sum\limits_{j=1}^7 
\Big(\Theta_{h_i} \omega_{h_j} - e_j \underbrace{\Theta_{h_i} \omega_{h_0}}_{=0}\Big) C_j.
$$
The underbraced expression $\Theta_{h_i} \omega_{h_0}$ vanishes, because we always have $\delta_{i0}=0$ since $i \neq 0$. 

In view of~(\ref{stern2}) the latter equation simplifies to 
$$
\sum\limits_{h=0}^7 \Big[\Theta_{h_i}\Big(\mu \cdot (\sum\limits_{j=0}^7 n_{h_j} \eta_j)  \Big)\cdot \lambda  \Big] - {\cal{N}}(\lambda) {\cal{N}}(\mu) \sum\limits_{h=0}^7 \Theta_{h_i} \eta_h = C_i \det(W).
$$
Thus, for $i=1,\ldots,7$ we obtain:
$$
C_i = \frac{\sum\limits_{h=0}^7 \Big[\Theta_{h_i}\Big(\mu \cdot (\sum\limits_{j=0}^7 n_{h_j} \eta_j)  \Big)\cdot \lambda  \Big] - {\cal{N}}(\lambda) {\cal{N}}(\mu) \sum\limits_{h=0}^7 \Theta_{h_i} \eta_h}{\det(W)}.
$$
Now we are in position to calculate the traces of the octonionic Weierstra{\ss}' functions. 

First of all we recall that 
$$
\zeta_{\lambda,\mu}(z) - \frac{1}{({\cal{N}}(\lambda){\cal{N}}(\mu))^3} \zeta(z) = 
\frac{1}{({\cal{N}}(\lambda){\cal{N}}(\mu))^3}  \sum\limits_{v \in {\cal{V}}_{\lambda,\mu;\Omega}\backslash\{0\}} \zeta(z+v) + \sum\limits_{j=1}^7 V_{\tau(j)}(z) C_j + C.
$$
Since $\zeta_{\lambda,\mu}$ is an odd function, $\zeta_{\lambda,\mu}$ is an odd function, too. Consequently, 
$$
\zeta_{\lambda,\mu}(z) - \frac{1}{({\cal{N}}(\lambda){\cal{N}}(\mu))^3} \zeta(z) = {\cal{O}}(z),
$$
around zero because the singular parts cancel out, following from Proposition~\ref{Prp49}. In view of $\lim\limits_{z \to 0}\wp_{\tau(i)}(z)-q_{\tau(i)}(z) = 0$ for all $i=1,\ldots,7$, which is clear from the series representation, because $\lim\limits_{z \to 0} q_{\tau(i)}(z+\omega)-q_{\tau(i)}(\omega)=0$ one even has that the expression on the right-hand side is of order ${\cal{O}}(z^3)$.

So, in particular, 
$$
\lim\limits_{z \to 0} \Big(\zeta_{\lambda,\mu}(z) - \frac{1}{({\cal{N}}(\lambda){\cal{N}}(\mu))^3} \zeta(z) \Big) = 0,
$$
and since $\lim\limits_{z \to 0} V_{\tau(j)}(z)=0$, one has 
$$
\frac{1}{({\cal{N}}(\lambda){\cal{N}}(\mu))^3} \sum\limits_{v \in {\cal{V}}_{\lambda,\mu;\Omega}\backslash\{0\}} \zeta(v) = -C. 
$$

Next, 
$$
\frac{\partial \zeta_{\lambda,\mu}}{\partial x_i} - \frac{1}{({\cal{N}}(\lambda){\cal{N}}(\mu))^3} \wp_{\tau(i)}
$$
is an even expression and of the form ${\cal{O}}(z^2)$ around the origin. So,  
$$
\lim\limits_{z \to 0}\Big(   
\frac{\partial \zeta_{\lambda,\mu}(z)}{\partial x_i} - \frac{1}{({\cal{N}}(\lambda){\cal{N}}(\mu))^3} \wp_{\tau(i)}(z)
\Big) = 0.
$$
On the other hand 
$$
\frac{\partial \zeta_{\lambda,\mu}(z)}{\partial x_i} - \frac{1}{({\cal{N}}(\lambda){\cal{N}}(\mu))^3} \wp_{\tau(i)}(z) = \frac{1}{({\cal{N}}(\lambda){\cal{N}}(\mu))^3} \sum\limits_{v \in {\cal{V}}_{\lambda,\mu;\Omega} \backslash\{0\}} \wp_{\tau(i)}(z+v) + \sum\limits_{j=1}^7 \frac{\partial }{\partial x_j} V_{\tau(i)}(z) C_j.
$$
 
Thus, with the same limit argument we obtain that 
$$
\frac{1}{({\cal{N}}(\lambda){\cal{N}}(\mu))^3}
\sum\limits_{v \in {\cal{V}}_{\lambda,\mu;\Omega} \backslash\{0\}} \wp_{\tau(i)}(v) = -C_i
$$

If we apply the formula for the expression of the constants $C_i$ derived above, then  we finally arrive at the desired trace formula for the octonionic CM-division values of the associated $\wp_{\tau(i)}$-function 
$$
\sum\limits_{v \in {\cal{V}}_{\lambda,\mu;\Omega}\backslash\{0\}} \wp_{\tau(i)}(v) 
= - \frac{({\cal{N}}(\lambda) {\cal{N}}(\mu)  )^3  }{\det(W)} \Bigg[
\sum\limits_{h=0}^7 \Theta_{h_i}\Big( \mu \cdot (\sum\limits_{j=0}^7 n_{h_j} \eta_j)\cdot \lambda
\Big)
- {\cal{N}}(\lambda) {\cal{N}}(\mu) \sum\limits_{h=0}^7\Theta_{h_i} \eta_h
\Bigg]
$$
for each $i=1,\ldots,7$, 
and the theorem is proved.
\end{proof}
\begin{remark}
For the sake of completeness we want to re-emphasize that the constant $C$ gives up to a minus sign multiplied with the triple of the quadratic norm expressions exactly the value of the trace of the CM-division values of the $\mathbb{O}$-regular Weierstra{\ss} $\zeta$-function:
$$
\sum\limits_{v \in {\cal{V}}_{\lambda,\mu;\Omega}\backslash\{0\}} \zeta(v) = - ({\cal{N}}(\lambda){\cal{N}}(\mu))^3 \cdot C. 
$$ 
Unfortunately, unlike for the constants $C_i$, we have no further information on the algebraic nature of the constant $C$ so far. 
\end{remark}

\subsection{Final remark and outlook}
The trace of the octonionic division values of the functions $\wp_{\tau(i)}$ is an octonion whose  real components are elements from the field that is generated by the algebraic number field of the real components of the primitive periods $\omega_h$ (canonically from a triquadratic number field) and by the real components of the octonionic Legendre constants $\eta_h$. It would be a great goal to figure out under which conditions these are algebraic. In that case, the division values of the associated octonionic elliptic functions could play a key role in the construction of algebraic number fields contributing in another way to Hilbert's twelfth problem from \cite{Hilbert}.      
 
Furthermore, it would be an essential question whether similar constructions could be carried over to the more general framework of Cayley-Dickson algebras or even more generally to graded deformed $\mathbb{R}_F \mathbb{Z}^n$-algebras that we discussed in \cite{AlbuquerqueKra2008}. As explained in Section 3 the theory of CM-lattices is available for this much more general setting. However, in the proofs and in the constructions of Section 4.3 we applied at several places the alternative property which is true in the octonionic case but not anymore beyond this when proceeding with the usual Cayley-Dickson doubling. Maybe some of these results can be carried over to the context of the algebra of the $2^n$-ons considered by D. Warren in \cite{WarrenDSmith}, indicating another possibility for further future research in this kind of direction. 

A possible field of application consists also in understanding whether these generalizations of elliptic functions may play a similar role in $G_2$-gauge theories, analogously to the role of quaternionic regular elliptic functions playing for $SU(2)$-instantons, cf. \cite{GTBook}.

\end{document}